\documentclass[12pt]{amsart}
\usepackage{amscd}

\numberwithin{equation}{section}

\newtheorem{thm}{Theorem}[section]
\newtheorem{theorem}{Theorem}
\newtheorem{lemma}[thm]{Lemma}

\newtheorem{proposition}[thm]{Proposition}

\theoremstyle{definition}

\newtheorem{remark}[thm]{Remark}

\newtheorem{definition}[thm]{Definition}
\newtheorem{claim}[thm]{Claim}

\newtheorem{defn-thm}[thm]{Definition-Theorem}

\usepackage{amsmath}
\usepackage{amsthm}
\usepackage{commath}
\usepackage{amssymb}
\usepackage{graphicx}
\usepackage{hyperref}
\begin{document}
\title[Non-collapsing maps and projections to Large convex sets]{Harmonic projections in negative curvature II: large convex sets}
\author{Ognjen To\v{s}i\'{c}}\address{
    Mathematical Institute\\
    University of Oxford\\
    United Kingdom
}
\maketitle
\begin{abstract}
    An important result in the theory of harmonic maps is due to Benoist--Hulin: given a quasi-isometry $f:X\to Y$ between pinched Hadamard manifolds, there exists a unique harmonic map at a finite distance from $f$. Here we show existence of harmonic maps under a weaker condition on $f$, that we call non-collapsing -- we require that the following two conditions hold uniformly in $x\in X$: (1) average distance from $f(x)$ to $f(y)$ for $y$ on the sphere of radius $R$ centered at $x$ grows linearly with $R$ (2) the pre-image under $f$ of small cones with apex $f(x)$ have low harmonic measures on spheres centered at $x$. Using these ideas, we also continue the previous work of the author on existence of harmonic maps that are at a finite distance from projections to certain convex sets. We show this existence in a pinched negative curvature setting, when the convex set is large enough. For hyperbolic spaces, this includes the convex hulls of open sets in the sphere at infinity with sufficiently regular boundary.
\end{abstract}
\section{Introduction}
A classical conjecture in the theory of harmonic maps is the Schoen conjecture, stating that for any quasi-isometry $f:\mathbb{H}^2\to\mathbb{H}^2$ of the hyperbolic plane $\mathbb{H}^2$, there exists a harmonic self-map of $\mathbb{H}^2$ at a bounded distance from $f$. This was shown by Markovi\'c \cite{markovic-schoen}, and there have since been numerous generalizations to spaces other than $\mathbb{H}^2$. Most notable results were obtained by Markovi\'c \cite{markovic-h3} (for 3-dimensional hyperbolic space $\mathbb{H}^3$), Lemm--Markovi\'c \cite{lemm-markovic} (for higher-dimensional hyperbolic spaces $\mathbb{H}^n$ for $n\geq 3$), Benoist--Hulin \cite{bh-rank-one} (for rank one symmetric spaces), and Benoist--Hulin \cite{Benoist2017HarmonicQM} (for pinched Hadamard manifolds).
\par Here we generalize the results of \cite{Benoist2017HarmonicQM} on pinched Hadamard manifolds, meaning simply connected complete Riemannian manifolds with sectional curvatures bounded between two negative constants, by weakening the quasi-isometry requirement on the map $f$. 
\par For a pinched Hadamard manifold $X$, we use $\mathrm{dist}(\cdot, \cdot)$ to refer to the path metric on $X$ induced by the Riemannian metric. We will denote the visual boundary at infinity of $X$ with $\partial_\infty X$. Let $B_R(x)$ be the ball of radius $R$ centered at $x$, and let $\sigma_{x,R}$ be the harmonic measure on $\partial B_R(x)$ as seen from $x$.
\begin{definition}\label{dfn:qia}
    A Lipschitz map $f:X\to Y$ between pinched Hadamard manifolds is non-collapsing if the following two conditions hold
    \begin{enumerate}
        \item there exist constants $c, R_0>0$, such that for any $x\in X, R>R_0$, we have 
        \begin{align*}
            \int_{\partial B_R(x)} \mathrm{dist}(f(x), f(y))d\sigma_{x,R}(y)\geq cR, 
        \end{align*}
        and 
        \item for any $\varepsilon>0$, there exist $\theta, R_0>0$ such that for any $x\in X, R>R_0$ and $\xi\in\partial_\infty Y$, we have 
        \begin{align*}
            \sigma_{x, R}\left(\{y\in \partial B_R(x):f(y)\neq f(x)\text{ and }\measuredangle_{f(x)}(\xi, f(y))<\theta\}\right)<\varepsilon,
        \end{align*}
        where $\measuredangle_a(b, c)$ denotes the angle at $a$ between the geodesics $[a, b]$, joining $a$ and $b$, and $[a, c]$, joining $a$ and $c$.
    \end{enumerate}
\end{definition}
\begin{theorem}\label{thm:main-qia}
    For any non-collapsing Lipschitz map $f:X\to Y$ between pinched Hadamard manifolds, there exists a harmonic map $h:X\to Y$ such that $\sup\mathrm{dist}(h, f)<\infty$.  
\end{theorem}
It is implicitly contained in the work of Benoist--Hulin that any Lipschitz quasi-isometry is non-collapsing, so Theorem \ref{thm:main-qia} does in fact generalize the Lipschitz case of \cite[Theorem 1.1]{Benoist2017HarmonicQM}. For completeness, we include the proof in \S\ref{subsec:qi-nc}.
The main novelty of Theorem \ref{thm:main-qia} relative to \cite[Theorem 1.1]{Benoist2017HarmonicQM} is our generalization of the ``interior estimate'' \cite[\S 4]{Benoist2017HarmonicQM}.
\par As another application of our generalized interior estimate, we study harmonic maps that are at a finite distance from a nearest-point projection to a convex set in a pinched Hadamard manifold. The study of such maps was initiated by the author in \cite{tosic}, where the main result states that, given a pinched Hadamard manifold $X$, and a set $S$ in the boundary at infinity $\partial_\infty X$ of $X$, such that $S$ has sufficiently low dimension, there exists a harmonic self-map of $X$ that is at a finite distance from the nearest-point projection to the convex hull of $S$. Here we prove an analogue of this result for convex sets that are sufficiently large.
\begin{definition}\label{dfn:admissible}
    A closed convex subset $C$ of a pinched Hadamard manifold $X$ is called admissible if there exists an angle $\theta$ and a distance $R_0$ with the following property. For any $x\in C, R>R_0$, there exists a point $\xi\in\partial_\infty X$ such that \[\partial B_R(x)\cap \mathrm{Cone}(x\xi, \theta)\subseteq\partial B_R(x)\cap C,\]
    where $\mathrm{Cone}(x\xi, \theta)=\{y\in X: \measuredangle_{x}(y, \xi)<\theta\}$.
\end{definition}
\begin{theorem}\label{thm:main-general}
    Let $C$ be an admissible closed convex subset of a pinched Hadamard manifold $X$. There exists a harmonic map $h:X\to X$ that is a finite distance away from the nearest-point retraction $r:X\to C$.
\end{theorem}
Note that nearest point projections are in general not non-collapsing, so Theorem \ref{thm:main-general} can not be derived directly from Theorem \ref{thm:main-qia}. As mentioned above, the key common ingredient in both Theorem \ref{thm:main-general} and Theorem \ref{thm:main-qia} is the generalized interior estimate.
\par A rich class of admissible convex sets in hyperbolic spaces $\mathbb{H}^n$ is provided by convex hulls of open sets in $\partial_\infty\mathbb{H}^n\cong \mathbb{S}^{n-1}$ with sufficiently regular boundary.
\begin{theorem}\label{thm:main-admissible}
    Let $U\subseteq\partial_\infty \mathbb{H}^n=\mathbb{S}^{n-1}$ be an open set with quasiconformal boundary. Then the convex hull of $U$ is admissible.
\end{theorem}
Here by quasiconformal boundary we mean that near any point $x\in\partial U$, there exists a local quasiconformal map that sends $U$ to $\mathbb{R}_{+}\times\mathbb{R}^{n-2}$ and $x$ to the origin.
\subsection{More precise results}
We will in fact prove a slightly stronger version of Theorem \ref{thm:main-qia}. 
\begin{definition}\label{dfn:qia-omega}
    Let $\omega:\mathbb{R}_+\to\mathbb{R}_+$ be a function such that $\omega(x)\to\infty$ and $\frac{\omega(x)}{x}\to 0$ as $x\to\infty$. Then a Lipschitz map $f:X\to Y$ is called $\omega$-weakly non-collapsing (weakly non-collapsing map with size function $\omega$) if the following two conditions hold
    \begin{enumerate}
        \item there exist constants $c, R_0>0$, such that for any $x\in X, R>R_0$, we have 
        \begin{align*}
            \int_{\partial B_R(x)} \mathrm{dist}(f(x), f(y))d\sigma_{x,R}(y)\geq cR, 
        \end{align*}
        and 
        \item for any $\varepsilon>0$, there exist $\theta, R_0>0$ such that for any $x\in X, R>R_0$ and $\xi\in\partial_\infty Y$, we have 
        \begin{align*}
            \sigma_{x, R}\left(\{y\in\partial B_R(x):\measuredangle_{f(x)}(\xi, f(y))<\theta\text{ and }\mathrm{dist}(f(x), f(y))\geq \omega(R)\}\right)<\varepsilon.
        \end{align*}
    \end{enumerate} 
    We call an $\omega:\mathbb{R}_+\to\mathbb{R}_+$ with $\omega(x)\to \infty$ and $\frac{\omega(x)}{x}\to 0$ as $x\to\infty$ a sublinear size function. A Lipschitz map is weakly non-collapsing if it is $\omega$-weakly non-collapsing for some sublinear size function $\omega$. 
\end{definition}
\begin{theorem}\label{thm:main-qia-omega}
    For any weakly non-collapsing Lipschitz map $f:X\to Y$, there exists a harmonic map $h:X\to Y$ such that $\sup \mathrm{dist}(h, f)<\infty$.
\end{theorem}
\begin{remark}\label{remark:non-collapsing}
    \begin{enumerate}
        \item Note that a non-collapsing map as in Definition \ref{dfn:qia} is a weakly non-collapsing map with any size function, so Theorem \ref{thm:main-qia} follows immediately from Theorem \ref{thm:main-qia-omega}.
        \item We will show below that, if $f$ is a weakly non-collapsing map, and $\tilde{f}$ is a Lipschitz map such that $\sup\mathrm{dist}(f, \tilde{f})<\infty$, then $\tilde{f}$ is also weakly non-collapsing (albeit with a different size function). In particular, the harmonic map obtained either from Theorem \ref{thm:main-qia} or Theorem \ref{thm:main-qia-omega} is weakly non-collapsing, but not necessarily with size function $0$.
        \item If $f$ is an $\omega$-weakly non-collapsing map, and $\tilde{\omega}\geq\omega$ is a sublinear size function, then $f$ is also an $\tilde{\omega}$-weakly non-collapsing. Thus the condition $\omega(x)\to\infty$ as $x\to\infty$ in Definition \ref{dfn:qia-omega} is superfluous, and is there merely for convenience.
    \end{enumerate} 
\end{remark}
Both Theorem \ref{thm:main-general} and \ref{thm:main-qia-omega} follow from our generalized interior estimate, stated below.

\begin{definition}\label{dfn:uniformly-inner}
    Let $\mathcal{F}$ be a family of smooth maps between pointed pinched Hadamard manifolds. Then $\mathcal{F}$ is uniformly non-collapsing if it is uniformly Lipschitz, if the domain and range of any function in $\mathcal{F}$ have uniformly bounded pinching constants, and if the following two conditions hold
    \begin{enumerate}
        \item There exist constants $c, R_0>0$, such that for any $f:(X,x)\to (Y,y)$ in $\mathcal{F}$ and any $R>R_0$, we have 
        \begin{align*}
            \int_{\partial B_R(x)} \mathrm{dist}(f(x), f(y))d\sigma_{x,R}(y)\geq cR,
        \end{align*}
        and 
        \item There exists a sublinear size function $\omega:\mathbb{R}_+\to\mathbb{R}_+$ such that for any $\varepsilon>0$, there exist $\theta>0, R_0>0$ such that, for any $f:(X,x)\to(Y,y)$ in $\mathcal{F}$ and $R>R_0$, and any $\xi\in\partial_\infty Y$, we have 
        \begin{align*}
            \sigma_{x, R}\left(\{y\in\partial B_R(x):\measuredangle_{f(x)}(\xi, f(y))<\theta\text{ and }\mathrm{dist}(f(x), f(y))\geq \omega(R)\}\right)<\varepsilon.
        \end{align*}
    \end{enumerate}
\end{definition}
\begin{theorem}[Generalized interior estimate] \label{thm:generalized-interior-estimate}
    Let $\mathcal{F}=\{f_n:(X_n, x_n)\to (Y_n, y_n):n=1,2,...\}$ be a uniformly non-collapsing family. Suppose $R_n$ is a sequence of positive real numbers with $R_n\to\infty$, and let $h_n:B_{R_n}(x_n)\to Y_n$ be a sequence of harmonic maps, such that the maximum of $\mathrm{dist}(h_n, f_n)$ is achieved at $x_n\in X_n$. Then $\sup_n \sup\mathrm{dist}(f_n, h_n)<\infty$. 
\end{theorem}
\subsection{Lipschitz quasi-isometries are non-collapsing}\label{subsec:qi-nc}
In this brief subsection, we outline why Lipschitz quasi-isometries are non-collapsing. The entire argument is essentially contained in \cite[\S 4.5]{Benoist2017HarmonicQM}. 
\par Let $f:X\to Y$ be a map between pinched Hadamard manifolds, with sectional curvatures between $-b^2$ and $-a^2$, such that 
\begin{align*}
    L^{-1}\mathrm{dist}(x,y)-A\mathrm{dist}(f(x), f(y))\leq L \mathrm{dist}(x,y),
\end{align*}
for some constants $L\geq 1$ and $A\geq 0$. Then $f$ is Lipschitz and the first condition in Definition \ref{dfn:qia} clearly holds, and the second condition follows immediately from the claim below and \cite[Theorem 1.1]{Benoist2020HarmonicMO}. 
\begin{claim}
    For any $\varepsilon>0$, there exists $R_0>0$ large enough such that the following holds. For any $x\in X$ and $R>R_0$, and any $y_1, y_2\in \partial B_R(x)$ such that $\measuredangle_{f(x)}(f(y_1), f(y_2))<\varepsilon$, we have $\measuredangle_x(y_1, y_2)\leq 4\varepsilon^{\frac{a}{Lb}}$.
\end{claim}
\begin{proof}Denote the Gromov product by $(a|b)_c:=\frac{1}{2}(\mathrm{dist}(a,c)+\mathrm{dist}(b,c)-\mathrm{dist}(a,b))$ for $a,b,c$ points in a pinched Hadamard manifold. Denote $\theta:=4\varepsilon^{\frac{a}{Lb}}$.
\par Let $y_1, y_2\in\partial B_R(x)$ be such that $\measuredangle_x(y_1, y_2)\geq \theta$. By \cite[Lemma 2.1.a]{Benoist2017HarmonicQM}, we see that $(y_1|y_2)_x\leq a^{-1}\log\frac{4}{\theta}$. Moreover, $(y_1|x)_{y_2}=(y_2|x)_{y_1}=\frac{1}{2}\mathrm{dist}(y_1, y_2)$. Note that 
\begin{align*}
    \cosh(a\mathrm{dist}(y_1, y_2))\geq 1+2\sinh^2 R_0\sin^2\frac{\theta}{2},
\end{align*}
by comparison to the hyperbolic plane, and hence (after possibly increasing $R_0$) we have $(f(y_1)| f(x))_{f(y_2)}, (f(y_2)|f(x))_{f(y_1)}\geq b^{-1}$ by \cite[Lemma 2.2]{Benoist2017HarmonicQM}. Again by \cite[Lemma 2.2]{Benoist2017HarmonicQM}, we see that $(f(y_1)|f(y_2))_{f(x)}\leq L(y_1|y_2)_x\leq L a^{-1}\log\frac{4}{\theta}$. Then $\measuredangle_{f(x)}(f(y_1), f(y_2))\geq \left(\frac{\theta}{4}\right)^{Lb/a}=\varepsilon$, as desired.
\end{proof}
\subsection{Organization and a brief outline}
Here we briefly describe the contents of each section in the paper.
\par In \S\ref{sec:deforming}, we show that any weakly non-collapsing Lipschitz map can be deformed to a smooth weakly non-collapsing map with bounds on the first two derivatives. This is achieved by using the same argument as in \cite[\S 3]{tosic}, that is in turn a slight generalization of the argument of Benoist--Hulin \cite[\S 2]{Benoist2017HarmonicQM}. In particular, here we merely verify that the property of being weakly non-collapsing is preserved under finite distance deformations (although the size function is not preserved). This is an important step, as the proofs of both Theorem \ref{thm:main-qia-omega} and Theorem \ref{thm:main-general} depend on computations of the Laplacian of the distance function, using the classical computation of Schoen--Yau \cite{SCHOEN1979361}. For this we need the underlying maps to be at least $C^2$, and moreover we need control on the tension field of the map that we are trying to deform to a harmonic map.
\par In \S\ref{sec:generalized-interior-estimate} we prove Theorem \ref{thm:generalized-interior-estimate}. The main technical result in this section is Lemma \ref{lm:fundamental-inequality}, that easily implies Theorem \ref{thm:generalized-interior-estimate}, and that we believe is of independent interest. Lemma \ref{lm:fundamental-inequality} is a more precise quantitative version of the ``interior estimate'' of \cite[\S 4]{Benoist2017HarmonicQM}. The proof of Theorem \ref{thm:generalized-interior-estimate} boils down to the observation that since $\mathrm{dist}(f_n(x_n), h_n(\cdot))$ is a subharmonic function, we have 
\begin{align*}
    \int_{\partial B_{R_n}(x_n)} \left(\mathrm{dist}(f_n(x_n), h_n(y)) - \mathrm{dist}(f_n(x_n), h_n(x_n))\right) d\sigma_{x_n, R_n}(y)\geq 0,
\end{align*}
followed by an estimate of the integrand on the left-hand side in the regime where $\mathrm{dist}(f_n(x_n), h_n(x_n))\to\infty$ as $n\to\infty$, and reach a contradiction, along the lines of \cite[\S 4]{Benoist2017HarmonicQM}. This section is the heart of the paper, and a more detailed outline can be found at the start of \S\ref{sec:generalized-interior-estimate}.
\par In \S\ref{sec:weakly-non-collapsing} we derive Theorem \ref{thm:main-qia-omega} from Theorem \ref{thm:generalized-interior-estimate}. Given the generalized interior estimate, this is similar to the arguments in \cite{Benoist2017HarmonicQM} or \cite{tosic}. The idea is to take an exhaustive sequence of nested balls $B_1\subset B_2\subset ...\subset X$, and let $h_n:B_n\to Y$ be the harmonic map that agrees with $f$ on $\partial B_n$. We then extract a limit of the $h_n$. Using Theorem \ref{thm:generalized-interior-estimate}, in combination with an appropriate boundary estimate \cite[Proposition 3.7]{Benoist2017HarmonicQM} (along with control on the second derivative of $f$, obtained in \S\ref{sec:deforming}), we show that $\sup_n\sup_{B_n}\mathrm{dist}(f, h_n)<\infty$. Given this bound, the Arzela--Ascoli theorem combined with some classical results on harmonic maps (namely Schauder estimates \cite{petersen} and Cheng's lemma \cite{cheng}), allows us to extract a limit of $h_n$, that gives the desired harmonic map at a finite distance from $f$. 
\par In \S\ref{sec:nearest-point} we show Theorem \ref{thm:main-general}. The overall strategy is similar to the proof of Theorem \ref{thm:main-qia-omega}. We still have an exhaustive sequence of nested balls $B_n$, with harmonic maps $h_n:B_n\to X$, and wish to prove $\sup_n\sup_{B_n}\mathrm{dist}(r, h_n)<\infty$. The proof of this bound is again naturally divided into two pieces: one follows from Theorem \ref{thm:generalized-interior-estimate} and the fact that $r$ is uniformly non-collapsing in a neighbourhood of the convex set $C$ (which follows from admissibility), and the other follows from the arguments in the previous paper of the author \cite[\S 4]{tosic}.
\par Finally in \S\ref{sec:admissible} we show Theorem \ref{thm:main-admissible}. We give here a brief outline of the proof. Firstly, it is easy to see that the only way admissibility can fail is along a sequence of points $x_i$ converging to the boundary at infinity $\partial_\infty\mathbb{H}^n$. If this sequence converges to a point in $U$, admissibility holds. Assume therefore that the sequence converges to a point $\xi$ in $\partial U$. We rescale by isometries $A_i$ of $\mathbb{H}^n$ to map $x_i$ to a fixed compact set. Near the point $\xi$, there is a locally defined quasiconformal map $f$ that straightens $U$. By rescaling $f$ by $A_i$ and extracting a limit by standard compactness properties of quasiconformal maps, we see that $A_i^{-1}U$ converges to an open set, which provides the desired contradiction. This key argument is contained in Lemma \ref{lm:boundary-analysis} shown in \S\ref{subsec:bdry-analysis}. We note here that the significance of the condition on quasiconformal regularity of the boundary is related to the work of Tukia--V\"ais\"al\"a \cite{Tukia19829uasiconformalEF}. 
\subsection{Acknowledgements} This paper is an exposition of a part of my PhD thesis at the University of Oxford, and as such would not be possible without the help of my advisor Vladimir Markovi\'c. I would also like to thank my thesis committee consisting of Melanie Rupflin and Peter Topping for their thoughtful comments on this work. Finally, I would like to thank the anonymous referee for numerous useful suggestions that have improved the quality of this paper.
\subsection*{Notation}
We write $A\lesssim B$ when there exists a constant $C>0$ that depends only on the pinching constants and dimension of the relevant pinched Hadamard manifolds, such that $A\leq CB$. We similarly write $A\gtrsim B$ when $B\lesssim A$, and $A\approx B$ when $A\lesssim B\lesssim A$. 
\par We collect below some pieces of notation that appear throughout the paper for the reader's convenience, 
\begin{itemize}
    \item[--] given a Riemannian manifold $M$, the distance function $\mathrm{dist}:M\times M\to\mathbb{R}_+=\{x\in\mathbb{R}: x\geq 0\}$ always refers to the path metric induced by the Riemannian metric on $M$,
    \item[--] we denote by $B_R(x)$ the ball of radius $R$ centered at $x$, under the metric given by $\mathrm{dist}$,
    \item[--] we denote by $\sigma_{x,R}$ the harmonic measure on the sphere $\partial B_R(x)$, as seen from $x$, i.e. the measure defined by the equality 
    \begin{align*}
        h(x)=\int_{\partial B_R(x)} h(y)d\sigma_{x,R}(y)
    \end{align*}
    for all bounded harmonic functions $h:B_R(x)\to \mathbb{R}$,
    \item[--] when $X$ is a pinched Hadamard manifold, we denote by $\partial_\infty X$ the visual boundary at infinity of $X$,
    \item[--] for $x, y\in X\cup\partial_\infty X$, we denote by $[x,y]$ the geodesic joining $x$ and $y$,
    \item[--] for $a\in X, b,c\in X\cup\partial_\infty X\setminus\{a\}$, we denote by $\measuredangle_a(b,c)$ the angle at $a$ between the geodesics $[a,b]$ and $[a, c]$,
    \item[--] for $x\in X, \xi\in X\cup\partial_\infty X\setminus\{a\}$ and $\theta>0$, we denote by $\mathrm{Cone}(x\xi, \theta)$ the set of points $y\in X\cup\partial_\infty X$ such that $\measuredangle_x(\xi, y)<\theta$,
    \item[--] we denote by $\mathbb{H}^n$ the $n$-dimensional hyperbolic space, and by $\partial_\infty\mathbb{H}^n=\mathbb{S}^{n-1}$ the $(n-1)$-dimensional sphere at infinity,
    \item[--] we denote by $\norm{f}_\infty$ the supremum of some function $f$ (if $f$ is a section of some vector bundle equipped with a natural metric, we still denote by $\norm{f}_\infty$ the supremum of the norm of $f$).
\end{itemize}
\section{Preliminaries on the geometric analysis of harmonic maps}
Here we collect some estimates on harmonic maps between pinched Hadamard manifolds. Our first result is due to Cheng \cite[equation (2.9)]{cheng} (a simplified version is stated in \cite[Lemma 3.4]{Benoist2017HarmonicQM}). Denote by $B_R(x)$ the metric ball of radius $R$ centered at $x$ belonging to some metric space.
\begin{lemma}[Cheng's lemma]\label{lm:cheng}
    Let $M, N$ be Hadamard manifolds with sectional curvatures between $-b^2$ and $0$. Then for any $R>\varepsilon>0$, there exists a constant $C$ that depends only on $\varepsilon, b, \dim M, \dim N$, such that for any harmonic map $h:B_R(x)\to N$ with $x\in M$, we have 
    \begin{align*}
        \norm{Dh}_{L^\infty(B_{R-\varepsilon}(x))}\leq C\;\mathrm{diam}\left(h\left(B_R(x)\right)\right).
    \end{align*}
\end{lemma}
Our second result follows from Schauder elliptic estimates \cite[Theorem 70, pp. 303]{petersen} for linear elliptic operators of second order. We want to apply these results to harmonic maps, that are solutions to a second order semilinear elliptic equation, so a slight modification is required. This modification is well-known, but we include a brief proof for completeness.
\begin{theorem}[Nonlinear Schauder elliptic estimates]\label{thm:schauder}
    Let $M, N$ be pinched Hadamard manifolds, and let $\Omega_0\subset\Omega\subset M$ be open sets with compact closures, such that $\bar{\Omega}_0\subset\Omega_1$. Suppose $h:\Omega\to N$ is a harmonic map with bounded image. Then for any $\alpha\in(0,1)$, we have 
    \begin{align*}
        \norm{h}_{C^{2,\alpha}(\Omega_0)}\leq C=C\left(\Omega, \Omega_0, N, \mathrm{diam}\left(h\left(\Omega\right)\right), \alpha\right)
    \end{align*}
\end{theorem}
\begin{proof}
    Let $B$ be a closed ball containing $h(\Omega)$ of radius comparable to $\mathrm{diam}\left(h(\Omega)\right)$. Let $\Psi:\mathrm{int}(B)\to \mathbb{R}^{\dim N}$ be an embedding with the properties 
    \begin{align*}
        \norm{D\Psi^{\pm 1}}_\infty, \norm{D^2\Psi^{\pm 1}}_\infty < c_0.
    \end{align*} 
    Such coordinates exist by \cite[Lemma 5.2]{Benoist2017HarmonicQM}, and here $c_0$ depends only on curvature bounds and dimension of $N$, and $\mathrm{diam}\left(h(\Omega)\right)$. We write the harmonic map equation in the coordinates given by $\Psi$. The Riemannian metric only depends on the first derivative of $\Psi^{-1}$, and the Christoffel symbols only on the first two derivatives of $\Psi^{-1}$, so in particular we obtain a pointwise bound on both. 
    \par Pick arbitrary local coordinates for $\Omega$. We denote by $\mu=1,2,...,\dim N$ indices that refer to coordinates on $N$, and by $i=1,2,...,\dim M$ indices that refer to coordinates on $M$. We also denote by $h^\mu_i$ the derivative in the $i$-th direction of the $\mu$-component of $h$, and by $\left(h_{ij}^\mu\right)_{i, j=1,2,...,\dim M}$ the second derivative of the $\mu$-component. The harmonic map equation is 
    \begin{align*}
        \Delta\left(h^\mu\right)+g^{ij}h_i^\nu h_j^\eta\Gamma^\mu_{\nu\eta}=0,
    \end{align*}
    where $g_{ij}$ is the Riemannian metric on $M$, $g^{ij}$ is its inverse, and $\Gamma^\mu_{\nu\eta}$ are Christoffel symbols on $N$. Note that by Lemma \ref{lm:cheng}, we have a bound on the derivative of $h$. Since $\Gamma^{\mu}_{\nu\eta}$ is bounded, and since the Laplacian is elliptic, by the standard Schauder estimates \cite[Theorem 70, pp. 303]{petersen} we get a bound on the $C^{2,\alpha}$-norm of $h$. 
\end{proof}
\section{Deforming to smooth maps}\label{sec:deforming} Our aim here is to show that any weakly non-collapsing map can be deformed to a smooth weakly non-collapsing map, with control on the first two derivatives. Note that from \cite[Lemma 3.1]{tosic}, any Lipschitz map can be deformed to a smooth map with first two derivatives bounded. The following proposition is thus the aim of this section. 
\begin{proposition}\label{prop:qia-deform}
    Let $\mathcal{F}$ be a uniformly non-collapsing family with size function $\omega$, let $D>0$, and let $\tilde{\mathcal{F}}$ be a uniformly Lipschitz family of maps between pointed pinched Hadamard manifolds. Assume that for any $\tilde{f}:(X,x)\to (Y,y)$ in $\tilde{\mathcal{F}}$, there exists a map $f:(X,x)\to(Y,y)$ in $\mathcal{F}$, such that $\sup_X \mathrm{dist}(f, \tilde{f})<D$. Then $\tilde{\mathcal{F}}$ is uniformly non-collapsing with size function $\tilde{\omega}+2D$.
\end{proposition}
\begin{proof}
    To check Definition \ref{dfn:uniformly-inner}(1), we write, for any $\tilde{f}\in\mathcal{F}$,
    \begin{align*}
        \int_{\partial B_R(x)}\mathrm{dist}(\tilde{f}(x), \tilde{f}(y))d\sigma_{x,R}(y)&\geq \int_{\partial B_R(x)}\left(\mathrm{dist}(f(x), f(y))-2D\right)d\sigma_{x, R}(y)\\
        &\geq cR-2D\geq \frac{c}{2}R
    \end{align*}
    for $R>\max(R_0, 4c^{-1}D)$, where $R_0, c$ are constants from Definition \ref{dfn:qia-omega}(1) for $\mathcal{F}$, and where $f\in\mathcal{F}$ is such that $\mathrm{dist}(f,\tilde{f})\leq D$. 
    \par It remains to show Definition \ref{dfn:uniformly-inner}(2). Fix an arbitrary $\varepsilon>0$, and let $\theta, R_0$ be as in Definition \ref{dfn:uniformly-inner}(2) for $\mathcal{F}$. We will make use of the following proposition on cones in negative curvature, shown in the next subsection.
    \begin{proposition}\label{proposition:moving-cone}
        For any $D, \theta>0$, there exist $\hat{D}, \hat{\theta}>0$ such that for any two points $x, y\in X$ at a distance at most $D$ and any $\xi\in\partial_\infty X$, we have 
        \begin{align*}
            N_D\left(\mathrm{Cone}(x\xi, \hat{\theta})\setminus B_{\hat{D}}(x)\right)\subseteq \mathrm{Cone}(y\xi, \theta).
        \end{align*}
    \end{proposition}
    We choose $\tilde{\theta}, \tilde{R}_0>0$ as in Proposition \ref{proposition:moving-cone}, so that 
    \begin{align*}
        N_D(\mathrm{Cone}(x\xi, \tilde{\theta})\setminus B_{\tilde{R}_0}(x))\subset\mathrm{Cone}(y\xi, \theta),
    \end{align*}
    for any $x, y\in X$ of distance at most $D$, and any $\xi\in\partial_\infty X$.
    \par Let $\tilde{f}\in\tilde{\mathcal{F}}$ now be arbitrary, and let $f\in\mathcal{F}$ be such that $\mathrm{dist}(f, \tilde{f})\leq D$. By choice of $\tilde{\theta},\tilde{R}_0$, we have 
    \begin{align*}
        N_D\left(\mathrm{Cone}(\tilde{f}(x)\xi, \tilde{\theta})\setminus B_{\tilde{R}_0}(\tilde{f}(x))\right)\subseteq\mathrm{Cone}(f(x)\xi, \theta).
    \end{align*}
    We now have, for $R$ large enough such that $\omega(R)>\tilde{R}_0-2D$, 
    \begin{align*}
        N_D\left(\mathrm{Cone}(\tilde{f}(x)\xi, \tilde{\theta})\setminus B_{\omega(R)+2D}(\tilde{f}(x))\right)&\subseteq N_D\left(\mathrm{Cone}(\tilde{f}(x)\xi, \tilde{\theta})\setminus B_{\tilde{R}_0}(\tilde{f}(x))\right) \\
        &\subseteq \mathrm{Cone}(f(x)\xi, \theta),
    \end{align*}
    and hence 
    \begin{align}
        N_D\left(\mathrm{Cone}(\tilde{f}(x)\xi, \tilde{\theta})\setminus B_{\omega(R)+2D}(\tilde{f}(x))\right)&\subseteq \mathrm{Cone}(f(x)\xi, \theta)\setminus B_{\omega(R)+D}(\tilde{f}(x)) \nonumber \\
        &\subseteq \mathrm{Cone}(f(x)\xi, \theta)\setminus B_{\omega(R)}(f(x)). \nonumber
    \end{align}
    Combined with the fact that $\tilde{f}^{-1}(S)\subset N_D(f^{-1}(S))$ for any $S\subseteq Y$, and Definition \ref{dfn:uniformly-inner}(2) for $f\in\mathcal{F}$, we see that 
    \begin{align*}
        \sigma_{x,R}\left(\{y\in\partial B_R(x):\measuredangle_{\tilde{f}(x)}(\xi, \tilde{f}(y))<\tilde{\theta}\text{ and }\mathrm{dist}(\tilde{f}(x), \tilde{f}(y))\geq\omega(R)+2D\}\right)<\varepsilon,
    \end{align*}
    for any $x\in X, \xi\in\partial_\infty X$ and $R$ sufficiently large (depending only on $\theta,R_0,\omega$). 
\end{proof}
\subsection{Moving the apex of a cone}
Here we show Proposition \ref{proposition:moving-cone}.
    We fix $D, \theta>0$. Let $\hat{D}$ (resp. $\hat{\theta}$) be an arbitrary positive constant, that we will freely increase (resp. decrease) over the course of the proof. By \cite[Proposition 5.4]{tosic}, it suffices to show
    \begin{align}\label{eq:prop-moving-cone-main-1}
        \mathrm{Cone}(x\xi, \hat{\theta})\setminus B_{\hat{D}}(x)\subseteq\mathrm{Cone}(y\xi, \theta).
    \end{align}
    \begin{remark}
        Note that in \cite{tosic}, the author works with the visual metric on $\partial_\infty X$, whereas here we are interested in the angle metric. It is classical that the two are H\"older equivalent, and the direction we need follows readily from Claim \ref{claim:def-angle} and \cite[\S 2.5]{bourdon1993actions}.
    \end{remark}
    \par Let $z\in \mathrm{Cone}(x\xi,\hat{\theta})\setminus B_{\hat{D}}(x)$ and let $w$ be the point on $x\xi$ closest to $z$. Our first assertion is that 
    \begin{align}\label{eq:footpoint-far}
        \mathrm{dist}(x, w)\geq \min\left(\hat{D}, a^{-1}\log\frac{1}{\hat{\theta}}\right)+O(1),
    \end{align}
    where $a>0$ is a constant such that $M$ has all sectional curvatures at most $-a^2$. By comparison with the hyperbolic plane for the triangle $xzw$, we see that 
    \begin{align}\label{eq:cmp-1}
        \mathrm{sinh}\left(a\mathrm{dist}(z, w)\right)\leq \sin\measuredangle_{x}(z, w)\sinh\left(a\mathrm{dist}(x, z)\right).
    \end{align}
    This in particular shows that 
    \begin{align}\label{eq:dist-estimate-sinh}
        \mathrm{dist}(z, w)\leq \max\left(0, \mathrm{dist}(x, z)+a^{-1}\log \measuredangle_x(z, w)\right) + O(1).
    \end{align}
    Therefore by the triangle inequality 
    \begin{align*}
        \mathrm{dist}(x, w)&\geq \mathrm{dist}(x, z)-\mathrm{dist}(z, w) \\ 
        &\geq\min\left(\mathrm{dist}(x, z), a^{-1}\log \frac{1}{\measuredangle_x(z,w)}\right)+O(1)
        \\ &\geq \min\left(\hat{D}, a^{-1}\log\frac{1}{\hat{\theta}}\right)+O(1),
    \end{align*}
    thus showing (\ref{eq:footpoint-far}). 
    \par Let $\delta$ be the Gromov constant of $X$ as a hyperbolic metric space. By (\ref{eq:footpoint-far}), since $\mathrm{dist}(x, y)\leq D$, by choosing $\hat{D}$ large enough and $\hat{\theta}$ small enough, we can arrange it so that $\mathrm{dist}(w, xy)>10\delta$. Thus, by considering the ideal triangle $x\xi y$, we see that $\mathrm{dist}(w, y\xi)\leq\delta$. Therefore \begin{align}\label{eq:diff-dist-bounded}\mathrm{dist}(z, y\xi)\leq \mathrm{dist}(z, x\xi)+\delta. \end{align} 
    Similarly to (\ref{eq:cmp-1}), by comparison to the hyperbolic plane, we see that 
    \begin{align*}
        \sinh(b\mathrm{dist}(z ,y\xi))\geq \sinh(b\mathrm{dist}(z, y))\sin\measuredangle_y(z,\xi)\gtrsim e^{b(\mathrm{dist}(x, z)-D)}\measuredangle_y(z, \xi).
    \end{align*}
    It follows from (\ref{eq:diff-dist-bounded}) that 
    \begin{align*}
        \measuredangle_y(z,\xi)\lesssim e^{b(\mathrm{dist}(z, x\xi)-\mathrm{dist}(x, z))},
    \end{align*}
    where we absorbed $e^{b(D+\delta)}$ into the implicit constant. Applying (\ref{eq:dist-estimate-sinh}), we get 
    \begin{align*}
        \measuredangle_y(z, \xi)&\lesssim \exp\left(\max\left(-b\mathrm{dist}(x, z), ba^{-1}\log\measuredangle_x(z, w)\right)\right)\\
        &\lesssim\exp\left(-\min\left(b\hat{D}, ba^{-1}\log\frac{1}{\hat{\theta}}\right)\right).
    \end{align*}
    By increasing $\hat{D}$ and decreasing $\hat{\theta}$ further, we can ensure that $\measuredangle_y(z, \xi)<\theta$. Since $z$ was arbitrary, and none of our constants or choices of $\hat{D}, \hat{\theta}$ depended on $z$, this concludes the proof of (\ref{eq:prop-moving-cone-main-1}).
\section{Generalized interior estimate}\label{sec:generalized-interior-estimate}
This section is devoted to proving Theorem \ref{thm:generalized-interior-estimate}, which follows from the technical Lemma \ref{lm:fundamental-inequality} below. 
\begin{lemma}\label{lm:fundamental-inequality}
    For any constants $a,b,r,\varepsilon,\delta>0$ and integer $n$, there exist positive real numbers $M, N$, such that the following holds. Suppose $X, Y$ are pinched Hadamard manifolds of dimension at most $n$ with pinching constants $-b^2\leq -a^2<0$. Let $x\in X$, and let $f:B_{r+\varepsilon}(x)\to Y$ and $h:B_{r+\varepsilon}(x)\to Y$ be a smooth and harmonic map, respectively, such that $\mathrm{dist}(h, f)$ achieves its maximum over $B_{r+\varepsilon}(x)$ at $x$. Then either 
    \begin{align}\label{eq:fundamental-inequality-1}
        \mathrm{dist}(h(x), f(x))\leq M \mathrm{diam}(f(B_r))+N,
    \end{align}
    or 
    \begin{align*}
        \int_{\partial B_r(x)} \min\left(a\rho(y), \log\frac{\pi}{\theta(y)}\right)d\sigma_{x,r}(y)\geq \frac{1}{2}\int_{\partial B_r(x)} a\rho(y)d\sigma_{x,r}(y)-\delta,
    \end{align*}
    where
    \begin{align*}
        \rho(y)=\mathrm{dist}(f(x), f(y))\text{ and }\theta(y)=\measuredangle_{f(x)}(h(x), f(y)).
    \end{align*}
\end{lemma}
The proof of Lemma \ref{lm:fundamental-inequality} is a quantitative version of the proof of the ``interior estimate'' \cite[\S 4]{Benoist2017HarmonicQM}. We first outline the proof of Lemma \ref{lm:fundamental-inequality} briefly. We divide the outline into three steps.
\begin{enumerate}
    \item We first observe that $\mathrm{dist}(f(x), h(\cdot))$ is a subharmonic function, so in particular 
    \begin{align}\label{eq:outline-subh-int}
        \int_{\partial B_r(x)} \left(\mathrm{dist}(f(x), h(y))-\mathrm{dist}(f(x), h(x))\right)d\sigma_{x, r}(y)\geq 0.
    \end{align}
    The entirety of the proof of Lemma \ref{lm:fundamental-inequality} is estimating the integrand on the left-hand side under the assumption that $\mathrm{dist}(h(x), f(x))$ is very large.
    \item If $\mathrm{dist}(h(x), f(x))=:D$ is large enough, we have 
    \begin{gather}
        \inf_{y\in B_r(x)}\mathrm{dist}(f(x), h(y))\geq \frac{D}{2},\label{eq:outline-inf-dist} \\
        \sup_{y\in B_r(x)}\measuredangle_{f(x)}(h(x), h(y))\leq C\exp\left({-\frac{a}{2}D}\right).\label{eq:outline-sup-angle}
    \end{gather}
    Inequality (\ref{eq:outline-inf-dist}) follows from the fact that $\mathrm{dist}(f(x), h(\cdot))$ is a positive subharmonic function defined on $B_r(x)$ that takes the value $D$ at the center $x$, and is bounded above by $D+2\mathrm{diam}(f(B_r(x)))$, along with a gradient bound on $\mathrm{dist}(f(x), h(\cdot))$ that follows from Cheng's lemma (see Claim \ref{claim:supinf-convexity} and (\ref{eq:inf-rhoh-estimate})). For $D$ large enough, $D^{-1}\mathrm{diam}(f(B_r(x)))$ is very small, which forces $\inf_{y\in B_r(x)}\mathrm{dist}(f(x), h(y))$ to be comparable to $D$. Inequality (\ref{eq:outline-sup-angle}) then follows from (\ref{eq:outline-inf-dist}) and Cheng's lemma.
    \item We then have the chain of inequalities 
    \begin{gather*}
        \mathrm{dist}(f(x), h(y))-\mathrm{dist}(f(x), h(x))\leq \mathrm{dist}(f(x), h(y)) - \mathrm{dist}(f(y), h(y))\\
        \leq 2a^{-1}\log\frac{1}{\measuredangle_{f(x)}(f(y), h(y))} - \mathrm{dist}(f(x), f(y)) + O(1).
    \end{gather*}
    The inequality in the first line follows from the fact that \[\mathrm{dist}(f(x), h(x))=\sup_{B_r(x)}\mathrm{dist}(h, f),\] and the inequality in the second line follows from the comparison of the triangle with vertices $f(x), f(y), h(y)$ with the hyperbolic plane. Plugging the final inequality into (\ref{eq:outline-subh-int}) along with the bound (\ref{eq:outline-sup-angle}) yields Lemma \ref{lm:fundamental-inequality}.
\end{enumerate}
\par We first show Theorem \ref{thm:generalized-interior-estimate} assuming Lemma \ref{lm:fundamental-inequality} below, and then we show Lemma \ref{lm:fundamental-inequality} in \S\ref{subsec:proof-of-lm-fundamental-inequality}.
\begin{proof}[Proof of Theorem \ref{thm:generalized-interior-estimate}]
    We assume that $\sup_{B_{R_n}(x_n)}\mathrm{dist}(f_n, h_n)\to\infty$, possibly after passing to a subsequence. Fix a large constant $R\geq 1$, that we will choose later, and pass to a subsequence such that $R_n>R$ for all $n$. Our proof strategy is to apply Lemma \ref{lm:fundamental-inequality} to $B_R(x_n)$. 
    \par Since $\sup_{B_R(x_n)}\mathrm{dist}(f_n, h_n)\to\infty$, we eventually have violation of (\ref{eq:fundamental-inequality-1}). Thus for large $n$, we have 
    \begin{align*}
        \int_{\partial B_R(x_n)} \min\left(a\rho_n(y), \log\frac{\pi}{\theta_n(y)}\right)d\sigma_{x_n, R}(y)\gtrsim R,
    \end{align*}
    where 
    \begin{align*}
        \rho_n(y)=\mathrm{dist}(f_n(x_n), f_n(y))\text{ and }\theta_n(y)=\measuredangle_{f_n(x_n)}(h_n(x_n), f_n(y)).
    \end{align*}
    Note that in this proof, we suppress the dependence of implicit constants on the constants of $\mathcal{F}$ coming from Definition \ref{dfn:uniformly-inner}. We observe that since $f_n$ are uniformly Lipschitz, we have $\rho_n(y)\lesssim R$.
    \par Let $\omega$ be the size function of $\mathcal{F}$. We now let \[S_n=\{y\in\partial B_R(x_n):\theta_n(y)\leq \pi e^{-a\omega(R)}\text{ and }\rho_n(y)\geq \omega(R)\}.\] Then 
    \begin{align*}
        R&\lesssim \int_{\partial B_R(x_n)} \min\left(a\rho_n(y), \log\frac{\pi}{\theta_n(y)}\right)d\sigma_{x_n, R}(y)\\ &\leq \int_{S_n} Rd\sigma_{x_n, R}+\int_{\partial B_R(x_n)\setminus S_n} a\omega(R) d\sigma_{x_n, R}\leq R\sigma_{x_n, R}(S_n)+a\omega(R).
    \end{align*}
    By sublinearity of $\omega$, we have $\sigma_{x_n, R}(S_n)\gtrsim 1$. However, observe that 
    \begin{align*}
        S_n=f_n^{-1}\left(\mathrm{Cone}\left(f_n(x_n)h_n(x_n), e^{-a\omega(R)}\right)\setminus B_{\omega(R)}(f_n(x_n))\right)\cap\partial B_R(x_n).
    \end{align*}
    Thus for $R$ large enough, depending on the constants of $\mathcal{F}$, we reach a contradiction with Definition \ref{dfn:uniformly-inner}(2) for $\mathcal{F}$.
\end{proof}
\subsection{Proof of Lemma \ref{lm:fundamental-inequality}} \label{subsec:proof-of-lm-fundamental-inequality}

\begin{figure}
    \includegraphics[width=\textwidth]{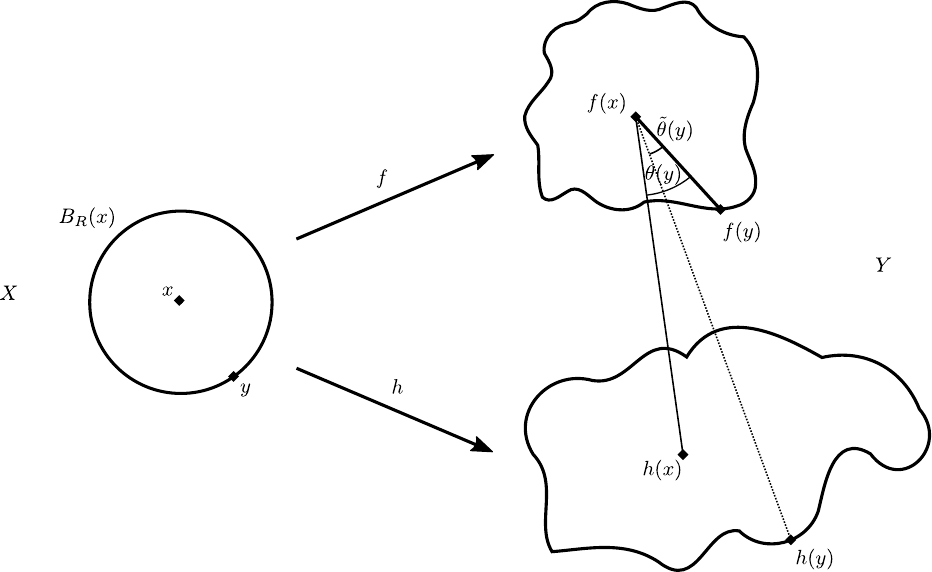}
    \caption{\label{fig:only} Setup of the proof of Lemma \ref{lm:fundamental-inequality}.}
\end{figure}
For clarity, we introduce the notation  
    \begin{gather*}
        \rho_f(y)=\mathrm{dist}(f(x), f(y)),\\
        \rho_h(y)=\mathrm{dist}(f(x), h(y)),\\
        \theta(y)=\measuredangle_{f(x)}(h(x), f(y)),\\
        \tilde{\theta}(y)=\measuredangle_{f(x)}(h(y), f(y)).
    \end{gather*}
    We will also denote by $\norm{\cdot}_\infty$ the $L^\infty$ norm over $B_{r+\varepsilon}(x)$. When referring to norms over smaller balls, we will specify it explicitly. The reader may wish to consult Figure \ref{fig:only} for the proof setup.
    \par The proof will follow from the following two inequalities. The first is 
    \begin{gather}
        \label{eq:visual-angle-bound}\measuredangle_{f(x)}(h(x), h(y))\leq Ca\frac{\rho_h(x)+\norm{\rho_f}_\infty}{\sinh\left(a\rho_h(x)-Ca\norm{\rho_f}_\infty\right)},
    \end{gather}
    for $y\in B_r(x)$, where $C=C(r,\varepsilon,a,b,n)>0$. The second is 
    \begin{gather}
        \label{eq:int-deficiency}\int_{\partial B_r(x)} \min\left(a\rho_f(y), \log\frac{\pi}{\tilde{\theta}(y)}\right)\geq \frac{1}{2}\int_{\partial B_r(x)} a\rho_f,
    \end{gather}
    provided $\rho_h(x)\geq (C+2)\norm{\rho_f}_\infty$. Here we are integrating against $\sigma_{x,r}$, but we drop the $d\sigma_{x,r}(y)$ in formulas for brevity.
    We first prove Lemma \ref{lm:fundamental-inequality} assuming inequalities (\ref{eq:visual-angle-bound}) and (\ref{eq:int-deficiency}), that we show in the next two subsections.
    \par Let $\xi=\frac{\pi}{4} \exp\left(-a\norm{\rho_f}_\infty\right)$, and 
    \begin{align*}
        \mathcal{C}=\{y\in \partial B_r: f(y)\neq f(x)\text{ and }\theta(y)\leq\xi\}.
    \end{align*} 
    By (\ref{eq:visual-angle-bound}), if $\rho_h(x)\geq M\norm{\rho_f}_\infty+N$ for some suitable constants $M, N$, we have 
    \begin{align*}
        \sup_{y\in \partial B_r(x)} \measuredangle_{f(x)}(h(x), h(y))\leq (1-e^{-\delta})\xi.
    \end{align*}
    We observe that for $y\in\partial B_r(x)\setminus\mathcal{C}$, we have 
    \begin{align*}
        \tilde{\theta}(y) \geq \theta(y)- \measuredangle_{f(x)}(h(x), h(y))\geq e^{-\delta}\theta(y).
    \end{align*}
    Thus for $y\in\partial B_r(x)\setminus\mathcal{C}$, we have 
    \begin{align*}
        \log \frac{\pi}{\tilde{\theta}(y)}\leq \delta+\log\frac{\pi}{\theta(y)}.
    \end{align*}
    We now estimate the integral on the left-hand side of (\ref{eq:int-deficiency}) by splitting the domain of integration,  
    \begin{align}\label{eq:a+b-estimate}
        \int_{\partial B_r(x)} \min\left(a\rho_f(y), \log\frac{\pi}{\tilde{\theta}(y)}\right)\leq \int_{\partial B_r(x)\setminus\mathcal{C}} \min\left(a\rho_f(y), \delta+\log\frac{\pi}{\theta(y)}\right)+\int_{\mathcal{C}} a\rho_f.
    \end{align}
    We note that by choice of $\xi$, for $y\in\mathcal{C}$, 
    \begin{align*}
        \delta+\log \frac{\pi}{\theta(y)}\geq \delta+\log\frac{\pi}{\xi}\geq a\norm{\rho_f}_\infty \geq a\rho_f(y).
    \end{align*}
    By (\ref{eq:a+b-estimate}) and (\ref{eq:int-deficiency}), we have 
    \begin{align*}
        \frac{1}{2}\int_{\partial B_r} a\rho_f\leq \int_{\partial B_r} \min\left(a\rho_f(y), \delta+\log\frac{\pi}{\theta(y)}\right),
    \end{align*}
    from which the result follows immediately.
    \subsubsection{Proof of (\ref{eq:visual-angle-bound}).}
    The proof of (\ref{eq:visual-angle-bound}) depends on the following estimate. 
    \begin{claim}\label{claim:supinf-convexity}
        Let $f:B_r(x)\to\mathbb{R}$ be a subharmonic $L$-Lipschitz function. Then there exists $\lambda=\lambda(r, L,a,b,n)$ with $0<\lambda<1$, such that 
        \begin{align*}
            f(x)\leq \lambda \inf_{B_r(x)}f+(1-\lambda)\sup_{B_r(x)}f.
        \end{align*}
    \end{claim}
    \begin{proof}
        In the proof, all bounds depend on $a,b,n$, and we suppress this in the notation.
        \par If $f$ is constant, there is nothing to prove. Therefore we post-compose $f$ with a linear function so that $\inf_{B_r(x)}f=0$ and $\sup_{B_r(x)}f=1$. Let $y\in B_r(x)$ be such that $f(y)=0$, and let $\rho=\mathrm{dist}(x, y)$. We will show the existence of $\lambda=\lambda(L)<1$ such that $f(x)\leq \lambda$. 
        \par For some $\theta=\theta(L, r)$, the following holds by comparison to the hyperbolic plane: given $y_1, y_2\in \partial B_{\tilde{r}}(x)$ for $\tilde{r}\leq r$ with $\measuredangle_x(y_1, y_2)<\theta$, we have $\mathrm{dist}(y_1, y_2)<\frac{1}{2L}$.
        \par We now analyze the inequality 
        \begin{align}\label{eq:supinf-convexity-int-subh}
            f(x)\leq\int_{\partial B_\rho(x)}f(y)d\sigma_{x, \rho}(y).
        \end{align}
        For any $z\in \partial B_\rho(x)\cap \mathrm{Cone}(xy, \theta)$, have $\mathrm{dist}(z, y)\leq \frac{1}{2L}$, and hence $\abs{f(z)}=\abs{f(z)-f(y)}\leq\frac{1}{2}$, since $f$ is $L$-Lipschitz. From (\ref{eq:supinf-convexity-int-subh}), we get 
        \begin{align*}
            f(x)&\leq \frac{1}{2}\sigma_{x,\rho}\left({\partial B_\rho(x)\cap\mathrm{Cone}(xy, \theta)} \right)+1-\sigma_{x,\rho}\left({\partial B_\rho(x)\cap\mathrm{Cone}(xy, \theta)} \right)\\
            &= 1-\frac{1}{2}\sigma_{x,\rho}\left({\partial B_\rho(x)\cap\mathrm{Cone}(xy, \theta)} \right).
        \end{align*}
        By \cite{Benoist2020HarmonicMO}, there is a $\mu=\mu(\theta)$ such that \[\sigma_{x,\rho}\left({\partial B_\rho(x)\cap\mathrm{Cone}(xy, \theta)} \right)\geq\mu.\] Therefore $f(x)\leq 1-\frac{1}{2}\mu$, and the claim is shown with $\lambda=\frac{1}{2}\mu$. 
    \end{proof}
    Since $\mathrm{dist}(h(y), f(y))\leq \rho_h(x)$, we have 
    \begin{align*}
        \rho_h(y)\leq \rho_f(y)+\mathrm{dist}(h(y), f(y))\leq\rho_f(y)+\rho_h(x),
    \end{align*}
    and hence $\norm{\rho_h}_{\infty}\leq \rho_h(x)+\norm{\rho_f}_{\infty}$. By Cheng's lemma, we have 
    \begin{align*}
        \norm{\nabla h}_{L^\infty(B_r(x))}\leq L(\varepsilon, a, b, n) \left(\rho_h(x)+\norm{\rho_f}_{L^\infty(B_{r+\varepsilon}(x))}\right).
    \end{align*}
    From Claim \ref{claim:supinf-convexity} applied to $\rho_h$ on $B_r(x)$, it follows that 
    \begin{align*}
        \rho_h(x)\leq \lambda \inf_{B_r} \rho_h+(1-\lambda) \left(\norm{\rho_f}_\infty+\rho_h(x)\right).
    \end{align*}
    In particular, we have 
    \begin{align}\label{eq:inf-rhoh-estimate}
        \inf_{B_r}\rho_h\geq \rho_h(x)-\frac{1-\lambda}{\lambda}\norm{\rho_f}_\infty\geq \rho_h(x)-C\norm{\rho_f}_\infty.
    \end{align}
    By comparison to the hyperbolic plane, we have for any $y\in B_r(x)$, 
    \begin{align*}
        a\mathrm{length}(h([x, y]))\geq \sinh(a\inf_{B_r}\rho_h) \measuredangle_{f(x)}(h(x), h(y)).
    \end{align*}
    By Cheng's lemma, 
    \[\mathrm{length}(h([x, y]))\leq r\norm{D h}_\infty\leq C\norm{\rho_h}_\infty\leq C\left(\rho_h(x)+\norm{\rho_f}_\infty\right).\] 
    Therefore, 
    \begin{align*}
        \measuredangle_{f(x)}(h(x), h(y))\leq Ca\frac{\rho_h(x)+\norm{\rho_f}_\infty}{\sinh\left(a\rho_h(x)-Ca\norm{\rho_f}_\infty\right)}.
    \end{align*}
    \subsubsection{Proof of (\ref{eq:int-deficiency}).} We first relate the deficiency (i.e. the slack in the triangle inequality) of the triangle with vertices $h(x),f(y), h(y)$ and the angle $\measuredangle_{f(x)}(h(y),f(y))$, that we remind the reader is denoted by $\tilde{\theta}(y)$. Note that a slight weakening of the following claim is stated in \cite[Lemma 2.1(b)]{Benoist2017HarmonicQM}.
    \begin{claim}\label{claim:def-angle}
        Let $D(y)=\rho_f(y)+\rho_h(y)-\mathrm{dist}(f(y), h(y))$. Then $\log\frac{\pi}{\tilde{\theta}(y)}\geq \frac{a}{2}D(y)$.
    \end{claim}
    \begin{proof}
        This follows from comparison with the hyperbolic plane. By the hyperbolic law of cosines, we have 
        \begin{align*}
            \cosh(a\mathrm{dist}(f(y), h(y)))&\geq \cosh(a\rho_f(y))\cosh(a\rho_h(y))-\sinh(a\rho_f(y))\sinh(a\rho_h(y))\cos\tilde{\theta}(y)\\
            &=\cosh(a(\rho_f(y)-\rho_h(y)))+2\sin^2\frac{\tilde{\theta}(y)}{2}\sinh(a\rho_f(y))\sinh(a\rho_h(y))
        \end{align*}
        Therefore 
        \begin{align*}
            \sin^2\frac{\tilde{\theta}(y)}{2}&\leq \frac{\sinh\left(\frac{a}{2}(\mathrm{dist}(f,h)+\rho_f-\rho_h)\right)\sinh\left(\frac{a}{2}(\mathrm{dist}(f,h)+\rho_h-\rho_f)\right)}{\sinh(a\rho_f)\sinh(a\rho_h)}\\
            &\leq \frac{\sinh\left(a\rho_f-\frac{a}{2}D(y)\right)}{\sinh(a\rho_f)}\frac{\sinh\left(a\rho_h-\frac{a}{2}D(y)\right)}{\sinh(a\rho_h)}\\
            &\leq e^{-aD(y)},
        \end{align*}
        where we used the inequality $\sinh(x-y)\leq e^{-y}\sinh(x)$ for $y\geq 0$ twice, and the facts that $D(y)\leq \rho_f(y)+\rho_h(y)-\abs{\rho_f(y)-\rho_h(y)}=2\min(\rho_f(y), \rho_h(y))$. Taking logarithms, we see that $\frac{a}{2}D(y)\leq \log\frac{1}{\sin\frac{\tilde{\theta}}{2}}\leq\log\frac{\pi}{\tilde{\theta}(y)}$,
        since $\sin(x)\geq\frac{2x}{\pi}$ for $0\leq x\leq \frac{\pi}{2}$ by concavity. 
    \end{proof}
    By (\ref{eq:inf-rhoh-estimate}), and the assumption that $\rho_h(x)\geq(C+2)\norm{\rho_f}_\infty$, we have $\inf_{B_r}(a\rho_h)\geq 2\norm{\rho_f}_\infty$. We observe that for $y\in \partial B_r(x)$, we have by Claim \ref{claim:def-angle}
    \begin{align*}
        \log \frac{\pi}{\tilde{\theta}(y)}\geq \frac{a}{2}D(y).
    \end{align*}
    Note that $D(y)=\rho_f(y)+\rho_h(y)-\mathrm{dist}(f(y),h(y))\leq 2\rho_f(y)$ by the triangle inequality, and hence $\frac{a}{2}D(y)\leq a\rho_f(y)$.
    Integrating this inequality over $\partial B_r(x)$, we get 
    \begin{align*}
        \int_{\partial B_r(x)} \min\left(a\rho_f(y), \log\frac{\pi}{\tilde{\theta}(y)}\right)&\geq \frac{1}{2}\int_{\partial B_r(x)} aD(y)\\
        &=\frac{1}{2}\int_{\partial B_r(x)} \left(a\rho_h(x)+a\rho_f(y)-a\mathrm{dist}(f(y), h(y))\right)\\
        &\geq \frac{1}{2}\int_{\partial B_r(x)} a\rho_f,
    \end{align*}
    where we used subharmonicity of $\rho_h$ in going from the first to the second line. 
    \section{Weakly non-collapsing maps}\label{sec:weakly-non-collapsing}
    In this section we prove Theorem \ref{thm:main-qia-omega}, that in turn immediately implies Theorem \ref{thm:main-qia}, as explained in Remark \ref{remark:non-collapsing}. The proof of Theorem \ref{thm:main-qia-omega} has four steps. 
    \begin{enumerate}
        \item By combining \cite[Lemma 3.1]{tosic} and Proposition \ref{prop:qia-deform}, it follows immediately that the map $f$ is at a finite distance from a map that is weakly non-collapsing with first two derivatives bounded. 
        \item We then construct harmonic maps $h_n$ on larger and larger balls $B_n$ that agree with $f$ on $\partial B_n$.
        \item The boundary estimate of Benoist--Hulin \cite[Proposition 3.7]{Benoist2017HarmonicQM}, stated below as Proposition \ref{prop:boundary-estimate}, then shows that in any finite distance neighbourhood of the boundary $\partial B_n$, the distance between $f$ and $h_n$ remains bounded. The generalized interior estimate Theorem \ref{thm:generalized-interior-estimate} shows that the distance between $f$ and $h_n$ remains bounded far from the boundary of $B_n$. 
        \item The limiting argument following Benoist--Hulin, \cite[\S 3.3]{Benoist2017HarmonicQM}, that we state and prove below as Proposition \ref{prop:limiting}, then shows that a limit of $h_n$ can be extracted to get a harmonic map at a finite distance from $f$. This argument follows from the Arzela--Ascoli theorem, combined with some classical results on harmonic maps: Schauder elliptic estimates and Cheng's lemma.
    \end{enumerate}
    \begin{proposition}\label{prop:boundary-estimate}
        Let $f:X\to Y$ be a smooth map between two pinched Hadamard manifolds with first two derivatives bounded. Let $x_0\in X, R>0$, and let $h:B_R(x_0)\to Y$ be a harmonic map that agrees with $f$ on $\partial B_R(x_0)$. Then there exists a constant $C$ that depends only on $\norm{Df}_\infty, \norm{D^2f}_\infty$, and the pinching constants of $X, Y$, such that 
        \begin{align*}
            \mathrm{dist}(h(x), f(x))\leq C\mathrm{dist}(x, \partial B_R(x_0))
        \end{align*}
    \end{proposition} 
    \begin{proposition}\label{prop:limiting}
        Let $X, Y$ be pinched Hadamard manifolds, and let $f:X\to Y$ be a smooth Lipschitz map with bounded second derivative. We fix a point $x\in X$, and let $h_n:B_n(x)\to Y$ be harmonic maps such that $\sup_n\sup_{B_n(x)} \mathrm{dist}(f, h_n)<\infty$. Then there exists a harmonic map $h:X\to Y$ such that $\sup\mathrm{dist}(h, f)<\infty$.  
    \end{proposition}
    \par Proposition \ref{prop:boundary-estimate} is exactly stated in \cite{Benoist2017HarmonicQM}, so we omit the proof. Proposition \ref{prop:limiting} appears in \cite{Benoist2017HarmonicQM} as well, however it was never explicitly stated, so for the reader's convenience we include the proof below in \S\ref{subsec:pf-prop-limiting}. We then prove Theorem \ref{thm:main-general} in \S\ref{subsec:pf-thm-main-qia-omega}.
    \subsection{Proof of Proposition \ref{prop:limiting}}\label{subsec:pf-prop-limiting}
    For any fixed compact set $K\subset X$, we have 
    \begin{align*}
        \mathrm{diam}(h_n(K))\leq 2\sup_{B_n(x)}\mathrm{dist}(h_n, f)+\mathrm{diam}(f(K))
    \end{align*}
    Hence $\mathrm{diam}(h_n(K))$ is bounded, and hence by Cheng's lemma \[\sup_{n>\mathrm{diam}(K)+1}\norm{D h_n}_{L^\infty(K)}<\infty.\]
    Therefore by the Arzela--Ascoli theorem, we may pass to a subsequence and extract a limit $h_n\to h$, that is uniform on compact subsets of $X$. 
    \par From the fact that $\sup_n\norm{D h_n}_{L^\infty(K)}<\infty$ for any compact set $K\subset X$, we see that for $\alpha\in(0,1)$, we have 
    \begin{align*}
        \sup_n\norm{h_n}_{C^\alpha(K)}<\infty
    \end{align*}
    From Schauder elliptic estimates (see Theorem \ref{thm:schauder}) and the fact that $h_n$ is harmonic, we see that $\sup_n \norm{h_n}_{C^{2,\alpha}(K)}<\infty$ for any compact $K\subset X$. Applying Arzela--Ascoli again, we may extract a further subsequence such that $D^2 h_n\to H$. It is easy to see that $H=D^2 h$, so in particular $h$ is harmonic. 
    \par Finally, we have 
    \begin{align*}
        \sup \mathrm{dist}(h, f)\leq\sup_n\sup_{B_n(x)}\mathrm{dist}(h_n, f)<\infty,
    \end{align*}
    which concludes the proof of Proposition \ref{prop:limiting}.
    \subsection{Proof of Theorem \ref{thm:main-qia-omega}}\label{subsec:pf-thm-main-qia-omega}
    Let $f:X\to Y$ be an $\omega$-weakly non-collapsing map between pinched Hadamard manifolds. By \cite[Lemma 3.1]{tosic} there exists a smooth $\tilde{f}:X\to Y$ such that $D\tilde{f}, D^2\tilde{f}$ are bounded and $\sup\mathrm{dist}(f, \tilde{f})<\infty$. Proposition \ref{prop:qia-deform} then guarantees that $\tilde{f}$ is a weakly non-collapsing map, possibly with a different size function.
    \par Fix an arbitrary point $x\in X$. Then let $h_n:B_n(x)\to Y$ be the harmonic map that agrees with $\tilde{f}$ on $\partial B_n(x)$. If $\sup \mathrm{dist}(\tilde{f}, h_n)$ is a bounded sequence, by Proposition \ref{prop:limiting}, we are done. Assume therefore, after passing to a subsequence, that $\sup\mathrm{dist}(\tilde{f}, h_n)\to\infty$. 
    \par Let $x_n\in B_n(x)$ be a sequence of points such that the maximum of $\mathrm{dist}(\tilde{f}, h_n)$ is achieved at $x_n$. By Proposition \ref{prop:boundary-estimate}, we have \[R_n=\mathrm{dist}(x_n, \partial B_n(x))\to\infty.\] 
    We observe that the family of maps $\{\tilde{f}:(X, x_n)\to (Y, f(x_n))\text{ for }n=1,2,...\}$ is uniformly non-collapsing by definition. Applying Theorem \ref{thm:generalized-interior-estimate} to the harmonic maps $h_n: B_{R_n}(x_n)\to Y$, we get that 
    \begin{align*}
        \sup_n\sup_{B_n(x)}\mathrm{dist}(\tilde{f}, h_n)=\sup_n\mathrm{dist}(\tilde{f}(x_n), h_n(x_n))<\infty,
    \end{align*} 
    which is a contradiction. 
    \section{Nearest-point projections to admissible convex sets}\label{sec:nearest-point}
    This section is devoted to showing Theorem \ref{thm:main-general}. We first give a rough outline of the proof. As in the proof of Theorem \ref{thm:main-qia-omega}, we construct harmonic maps $h_n$ defined on larger and larger balls $B_n(o)$ for some fixed $o\in X$, agreeing with $r$ on the boundaries $\partial B_n(o)$. The goal is to use the limiting argument in Proposition \ref{prop:limiting} to get a harmonic map defined on all of $X$. It therefore suffices to show that $\sup_X\mathrm{dist}(h_n, r)$ is a bounded sequence. We do this in two steps.
    \begin{enumerate}
        \item We first show that for some fixed $D>0$, we have \[\sup_{X\setminus N_D(C)}\mathrm{dist}(h_n, r)\leq\sup_{N_D(C)}\mathrm{dist}(h_n, r)+O(1).\] 
        This inequality is derived analogously to \cite[\S 4]{tosic}, and it follows from the existence of a bounded subharmonic function $\Phi$ such that $\Delta\Phi\gtrsim e^{-a\mathrm{dist}(\cdot, C)}$ on $X\setminus N_D(C)$, for some fixed $D>0$, and from the classical inequality of Schoen--Yau \cite{SCHOEN1979361} on the Laplacian of the distance between two functions.
        \item It therefore remains to show that $\sup_{N_D(C)}\mathrm{dist}(h_n, r)$ is bounded. This follows from our generalized interior estimate Theorem \ref{thm:generalized-interior-estimate}, since the map $r$ is non-collapsing near the convex set $C$. This bound is contained in Proposition \ref{prop:interior-estimate} below.
    \end{enumerate}
    
    To state Proposition \ref{prop:interior-estimate}, we first note that, given any admissible convex set $C$ and a nearest-point projection map $r:X\to C$, by \cite[Corollary 3.7]{tosic}, there exists a smooth map $\tilde{r}:X\to X$ such that $\mathcal{D}:=\sup_X\mathrm{dist}(r, \tilde{r})<\infty$, and 
    \begin{gather*}
        \norm{D\tilde{r}}\lesssim e^{-a\mathrm{dist}(\cdot, C)},\\
        \norm{\tau(\tilde{r})}\lesssim e^{-a\mathrm{dist}(\cdot, C)}.
    \end{gather*}
    \begin{proposition}\label{prop:interior-estimate}
        Let $D>0$. There exist constants $R_0=R_0(D)>0$ and $M=M(D)>0$, such that, for any $x\in N_D(C)$ and $R>R_0$ we have the following property. Given a harmonic map $h:B_R(x)\to X$ such that $\mathrm{dist}(h, \tilde{r})$ achieves its maximum at $x$, we have $\mathrm{dist}(h, \tilde{r})<M$.
    \end{proposition}

    We first show Theorem \ref{thm:main-general} assuming Proposition \ref{prop:interior-estimate} in \S\ref{subsec:pf-thm-main-general}. We then show Proposition \ref{prop:interior-estimate} in \S\ref{subsec:prop-interior-estimate}.
    \subsection{Proof of Theorem \ref{thm:main-general}}\label{subsec:pf-thm-main-general}
     From \cite[Proposition 4.4]{tosic}, for some $D>0$ large enough, there exist subharmonic functions $\phi_n:X\to\mathbb{R}$ for $n>D-1$, such that 
    \begin{align*}
        \Delta\phi_n\geq 1\text{ on }N_{n+1}(C)\setminus N_n(C),
    \end{align*}
    and $\sup_n\norm{\phi_n}_\infty<\infty$. We now construct the function 
    \begin{align*}
        \Phi=\sum_{n=\lfloor D\rfloor}^\infty e^{-an}\phi_n,
    \end{align*}
    such that $\Phi$ is a bounded subharmonic function, with the property that $\Delta\Phi\gtrsim e^{-a\mathrm{dist}(\cdot, C)}$ on $X\setminus N_D(C)$.
    \par We now fix an arbitrary point $o\in X$, and let $h_N:B_N(o)\to X$ be the harmonic map that agrees with $\tilde{r}$ on $\partial B_N(o)$. By Proposition \ref{prop:limiting}, it suffices to show the following claim.
    \begin{claim}
        The sequence $\sup_{B_N(o)}\mathrm{dist}(h_N, \tilde{r})$ is bounded.
    \end{claim}
    \begin{proof}
        Assume that, possibly after passing to a subsequence, we have \[\sup_{B_N(o)}\mathrm{dist}(h_N, \tilde{r})\to\infty.\] 
        Note that from \cite{SCHOEN1979361}, we have 
        \begin{align*}
            \Delta\mathrm{dist}(h_N, \tilde{r})\gtrsim -\norm{\tau(\tilde{r})}\gtrsim -e^{-a\mathrm{dist}(\cdot, C)}.
        \end{align*}
        Therefore, for a suitably chosen constant $c$, the function
        \begin{align*}
            \mathrm{dist}(h_N, \tilde{r})+c\Phi 
        \end{align*}
        is subharmonic on $X\setminus N_D(C)$. 
        \par Let $x_N\in B_N(o)$ be the point where the maximum of $\mathrm{dist}(h_N, \tilde{r})$ is achieved. If $x_N\in X\setminus N_D(C)$ for infinitely many $N$, then 
        \begin{align*}
            \mathrm{dist}(h_N(x_N), \tilde{r}(x_N))+c\Phi(x_N)\leq \sup_{\partial B_N(o)}(\mathrm{dist}(h_N, \tilde{r})+c\Phi)\lesssim \norm{\Phi}_\infty\lesssim 1,
        \end{align*}
        which is a contradiction since $\Phi$ is bounded. Thus, for infinitely many $N$, we have $x_N\in N_D(C)$. Proposition \ref{prop:boundary-estimate} shows that 
        $\mathrm{dist}(x_N, \partial B_N(o))\to\infty$ as $N\to\infty$. In particular, for $N$ large enough, we may apply Proposition \ref{prop:interior-estimate}, to get that $\sup_N\sup_{B_N(o)}\mathrm{dist}(h_N, \tilde{r})<\infty$. This is a contradiction.
    \end{proof}
    \subsection{Proof of Proposition \ref{prop:interior-estimate}}\label{subsec:prop-interior-estimate}
    This follows immediately from Theorem \ref{thm:generalized-interior-estimate}, once we show that the family \[\{\tilde{r}:(X,x)\to (X, \tilde{r}(x))\text{ for }x\in N_D(C)\}\] is uniformly non-collapsing. Recall that $\tilde{r}$ is the smooth approximation of the nearest-point projection $r:X\to C$. By Proposition \ref{prop:qia-deform}, it suffices to show that the family
    \begin{align*}
        \{r:(X, x)\to (X, r(x))\text{ for }x\in N_D(C)\}
    \end{align*}
    is uniformly non-collapsing. The rest of this subsection is devoted to showing this.
    \par We first check Definition \ref{dfn:uniformly-inner}(1). Let $\theta, R_0$ be as in Definition \ref{dfn:admissible}. 
     Fix some $x\in N_D(C)$, and set $\rho(y)=\mathrm{dist}({r}(x), {r}(y))$. Let $\hat{x}\in C$ be such that $\mathrm{dist}(x, \hat{x})\leq 2D$.
    From Definition \ref{dfn:admissible}, we see that there exists some $\xi\in\partial_\infty X$, such that $\partial B_R(\hat{x})\cap\mathrm{Cone}(\hat{x}\xi, \theta)\subseteq \partial B_R(\hat{x})\cap C$ for all $R>R_0$. By Proposition \ref{proposition:moving-cone}, we have 
    \begin{align*}
        \mathrm{Cone}(x\xi, \hat{\theta})\cap\partial B_R(x)\subset \mathrm{Cone}(\hat{x}\xi, \theta)\setminus B_{R-2D}(\hat{x})\subset C,
    \end{align*}
    for $R>\max(\hat{D}, R_0+D)$, where $\hat{\theta},\hat{D}$ are the constants from Proposition \ref{proposition:moving-cone}. Then we have 
    \begin{align*}
        \int_{\partial B_R(x)} \rho(y)&\geq \int_{\partial B_R(x)\cap C} \rho(y)=\int_{\partial B_R(x)\cap C}\mathrm{dist}({r}(x), y)\\ 
        &\geq \sigma_{x,R}(\mathrm{Cone}(x\xi, \hat{\theta})\cap \partial B_R(x)) (R-\mathrm{dist}(x, \tilde{r}(x)))\\
        &\gtrsim R-D\approx R,
    \end{align*}
    where we used the fundamental estimate of Benoist--Hulin \cite{Benoist2020HarmonicMO} that $\sigma_{x, R}(\mathrm{Cone}(x\xi, \hat{\theta})\cap\partial B_R(x))\gtrsim 1$, and where we assumed $R>2D$. 
    \par We now turn to Definition \ref{dfn:uniformly-inner}(2).
    \begin{claim}\label{claim:angle-bound}
        Let $w, y\in C$ be such that $\mathrm{dist}(w, y)=R$. Then for any $z\in r^{-1}(y)$, we have $\measuredangle_{w}(y, z)\leq\pi e^{-aR}$.
    \end{claim}
    \begin{proof}
        We define a comparison triangle $\bar{y}\bar{z}\bar{w}$ in the hyperbolic plane with the properties 
        \begin{gather*}
            \measuredangle_{\bar{y}}(\bar{z}, \bar{w})=\measuredangle_{y}(z, w),\text{ and } \\ 
            \mathrm{dist}(y,z)=\mathrm{dist}(\bar{y},\bar{z})\text{ and }\mathrm{dist}(y,w)=\mathrm{dist}(\bar{y},\bar{w}).
        \end{gather*}
        Then $\measuredangle_w(y,z)\leq\measuredangle_{\bar{w}}(\bar{y},\bar{z})$, so it suffices to estimate $\measuredangle_{\bar{w}}(\bar{y}, \bar{z})$. Note that $\measuredangle_{\bar{y}}(\bar{z}, \bar{w})\geq\measuredangle_y(z,w)\geq\frac{\pi}{2}$, so there exists a point $\bar{x}$ on the segment $\bar{z}\bar{w}$ such that $\measuredangle_{\bar{y}}(\bar{w},\bar{x})=\frac{\pi}{2}$. 
        Applying the dual hyperbolic law of cosines to the triangle $\bar{x}\bar{y}\bar{w}$ at the vertex $\bar{x}$, we see that 
        \begin{align*}
            \cos\measuredangle_{\bar{x}}(\bar{y}, \bar{w})=\sin\measuredangle_{\bar{w}}(\bar{y},\bar{x})\cosh(aR)\geq \frac{2\measuredangle_{\bar{w}}(\bar{y}, \bar{x})}{\pi}\frac{e^{aR}}{2}.
        \end{align*}
        Thus $\measuredangle_{\bar{w}}(\bar{y}, \bar{z})=\measuredangle_{\bar{w}}(\bar{y}, \bar{x})\leq \pi e^{-aR}$, concluding the proof.
    \end{proof}
    We now fix arbitrary $x\in N_D(C), \xi\in\partial_\infty X$ and $M>0$. Note that for any $y\in X$ such that $r(y)\in\mathrm{Cone}(r(x)\xi,\theta)\setminus B_M(r(x))$, applying Claim \ref{claim:angle-bound} to the points $r(x),r(y),y$, we see that $\measuredangle_{r(x)}(\xi, y)\leq\measuredangle_{r(x)}(\xi, r(y))+\pi e^{-aM}\leq \theta+\pi e^{-aM}$. In other words,
    \begin{align*}
        r^{-1}\left(\mathrm{Cone}(r(x)\xi, \theta)\setminus B_{M}(r(x))\right)\subseteq \mathrm{Cone}(r(x)\xi, \theta+\pi e^{-aM}),
    \end{align*}
    and hence in particular 
    \begin{align*}
        \partial B_R(x)\cap r^{-1}\left(\mathrm{Cone}(r(x)\xi,\theta)\setminus B_{\sqrt{R}}(r(x))\right)&\subseteq \partial B_R(x)\cap \mathrm{Cone}(r(x)\xi, \theta+\pi e^{-a\sqrt{R}})\\ 
        &\subseteq \partial B_R(x)\cap \mathrm{Cone}(x\xi,\tilde{\theta}(R, \theta)),
    \end{align*}
    where $\tilde{\theta}(R, \theta)\to 0$ as $R\to\infty, \theta\to 0$. Here in going from the first to the second line, we used Proposition \ref{proposition:moving-cone}. From the work of Benoist--Hulin \cite{Benoist2020HarmonicMO}, we see that $\sigma_{x,R}(\partial B_R(x)\cap \mathrm{Cone}(x\xi, \tilde{\theta}))\to 0$ as $\theta\to 0, R\to\infty$. Therefore Definition \ref{dfn:uniformly-inner}(2) holds, and Proposition \ref{prop:interior-estimate} is shown.
    \section{Admissible convex sets in hyperbolic spaces}\label{sec:admissible}
    In this section we prove Theorem \ref{thm:main-admissible}, that readily follows from the lemma below. For a set $S\subseteq\partial_\infty\mathbb{H}^n$, denote by $\mathrm{CH}(S)$ the closed convex hull of $S$.
    \begin{lemma}\label{lm:boundary-analysis}
        Let $S\subseteq\mathbb{S}^{n-1}$ be an open set with quasiconformal boundary. Then for any $D>0$, there exists an angle $\theta=\theta(D)>0$, such that for any $x\in N_D(\mathrm{CH}(S))$, there exists $\xi\in S$ such that $\mathrm{Cone}(x\xi, \theta)\cap\mathbb{S}^{n-1}\subseteq S$.
    \end{lemma}
    We prove Lemma \ref{lm:boundary-analysis} in \S\ref{subsec:bdry-analysis}. Theorem \ref{thm:main-admissible} then follows by simple hyperbolic geometry, that we explain in \S\ref{subsec:pf-thm-main-admissible}.
    \subsection{Boundary analysis: Proof of Lemma \ref{lm:boundary-analysis}}\label{subsec:bdry-analysis}
    Suppose that the conclusion of Lemma \ref{lm:boundary-analysis} fails. Then there exists a sequence $x_i\in\mathbb{H}^n$ such that $\sup_i \mathrm{dist}(x_i, \mathrm{CH}(S))<\infty$, and such that 
    \begin{align}\label{eq:contrary-assumption}
        \sup\{\theta:\mathrm{Cone}(x_i\xi, \theta)\cap\mathbb{S}^{n-1}\subseteq S\text{ for some }\xi\in S\}\to 0
    \end{align}
    as $i\to\infty$. Note that if $x_i$ remain in some compact set, after passing to a subsequence, we may assume that $x_i\to x_\infty\in X$. Since $S$ is an open set, this is a contradiction with (\ref{eq:contrary-assumption}). Assume therefore that $x_i\to s\in\partial_\infty X$, possibly after passing to a subsequence. Since $\sup_i \mathrm{dist}(x_i, \mathrm{CH}(S))<\infty$, we have $s\in \bar{S}$.
    If $s\in S\setminus\partial S$, then for any positive $\theta$, we have $\mathrm{Cone}(x_i s,\theta)\cap\mathbb{S}^{n-1}\subseteq S$ for all $i$ large enough, contradicting (\ref{eq:contrary-assumption}). 
    \par Therefore assume $s\in\partial S$. Let $U$ be an open set containing $s$, and $f:U\to V\subseteq\mathbb{R}^{n-1}$ be a quasiconformal homeomorphism, such that $f(s)=0$ and 
        \begin{align*}
            f(S\cap U)=V\cap(\mathbb{R}_{+}\times\mathbb{R}^{n-2}).
        \end{align*}
    Fix a point $o\in\mathbb{H}^n$, and let $A_i$ be an isometry of $\mathbb{H}^n$ with $A_i(o)=x_i$ and $A_i(s)=s$. Fix a point $y\in \mathbb{S}^{n-1}\setminus {U}$, and pass to a subsequence such that $A_i^{-1}(y)\to \hat{y}\in\mathbb{S}^{n-1}$. Note that since $x_i\to s$, by construction we have $A_i^{-1}(\mathbb{S}^{n-1}\setminus U)\to \{\hat{y}\}$ in the sense of Hausdorff distance.
    \par Now consider $f_i=B_i\circ f\circ A_i$, where $B_i$ is an isometry of $\mathbb{H}^n$ with $B_i\circ f_i\circ A_i(x)=x$ for $x\in \{s, u, v\}$, where $u,v\in\mathbb{S}^{n-1}\setminus \{s, \hat{y}\}$ are distinct arbitrarily chosen points. By standard compactness results on quasiconformal mappings (see e.g. \cite[Chapter 14, Theorems 3.1 and 3.2]{2014}), after passing to a subsequence, we may assume that $f_i$ converge uniformly on compact sets to a quasiconformal embedding $\hat{f}:\mathbb{S}^{n-1}\setminus\{\hat{y}\}\to\mathbb{S}^{n-1}$. Then $\hat{f}$ extends to a quasiconformal map $\mathbb{S}^{n-1}\to\mathbb{S}^{n-1}$, that we also denote by $\hat{f}$ (see \cite[Chapter 14, Theorem 10.6]{2014} or \cite[Chapter 14, Theorem 8.6]{2014}). 
    \par Let $G_i=B_i(\mathbb{R}_+\times\mathbb{R}^{n-2})$. Thus $f_i$ maps $A_i^{-1}(S\cap U)$ to $G_i$. 
    Define $\mathrm{GH}(A)$ of a set $A\subseteq\mathbb{S}^{n-1}$ to be the union of all geodesics with both endpoints in $A$. 
    \begin{claim}\label{claim:dist-to-gh}
        We have $\sup_i \mathrm{dist}(x_i, \mathrm{GH}(S\cap U))<\infty$.
    \end{claim}
    We first finish the proof assuming Claim \ref{claim:dist-to-gh}. Note that by Claim \ref{claim:dist-to-gh}, for any $i$, there exist $z_i, t_i\in S\cap U$ with the property that \begin{align}\label{eq:supi}\sup_i\mathrm{dist}(x_i, [z_i, t_i])<\infty. \end{align} After passing to a subsequence, we may assume that $A_i^{-1}(z_i)\to z, A_i^{-1}(t_i)\to t$. By (\ref{eq:supi}), we see that $z\neq t$. Observe that $G_i$ is a sequence of hemispheres in $\mathbb{S}^{n-1}$ such that $f_i(A_i^{-1}(z_i)), f_i(A_i^{-1}(t_i))\in G_i$. Note that $f_i(A_i^{-1}(z_i))\to \hat{f}(z)$ and $f_i(A_i^{-1}(t_i))\to \hat{f}(t)$. Since $z\neq t$, it follows that $\hat{f}(z)\neq\hat{f}(t)$, and thus there exists a hemisphere $G$ such that $G_i\to G$.
    \par We now finish the proof of Lemma \ref{lm:boundary-analysis}. Since $\hat{f}$ is a quasiconformal map, there exist a $\xi\in\mathbb{S}^{n-1}$ and $\theta>0$ such that, for $D=\mathrm{Cone}(o\xi,\theta)\cap\mathbb{S}^{n-1}$, we have $\hat{y}\not\in D$, and $f(D)\subset G$.
    \par For $i$ large enough, it follows that $D\subset A_i^{-1}(U)$, and that $f_i(D)\subset G_i=B_i(\mathbb{R}_+\times\mathbb{R}^{n-2})$. Thus $A_i(D)\subset U$ and $f(A_i(D))\subset\mathbb{R}_+\times\mathbb{R}^{n-2}$. Therefore $\mathrm{Cone}(x_iA_i(\xi), \theta)=A_i(D)\subset S\cap U$, which is a contradiction with (\ref{eq:contrary-assumption}). 
    \begin{proof}[Proof of Claim \ref{claim:dist-to-gh}]
        By \cite[Claim 5.3]{tosic}, we have $\mathrm{CH}(S)\subseteq N_{C}(\mathrm{GH}(S))$. Thus $K:=\sup_i \mathrm{dist}(x_i, \mathrm{GH}(S))<\infty$. Let $\tilde{U}\subset\hat{U}\subset U$ be open disks centered at $s$, such that $S\cap U\setminus \hat{U}\neq\emptyset$. Let $s_0\in S\cap U\setminus \hat{U}$ be arbitrary.
        \par For all $i$ larger than some $N_0$, we have $\mathrm{dist}(x_i, \mathrm{CH}(\partial_\infty\mathbb{H}^n\setminus\tilde{U}))\geq 2K$. For any such $i$, we therefore have $B_K(x_i)\cap \mathrm{CH}(\partial_\infty\mathbb{H}^n\setminus\tilde{U})=\emptyset$. In particular, if $\mathrm{dist}(x_i, [\xi,\eta])\leq K$, we must have $\xi\in \tilde{U}$ or $\eta\in\tilde{U}$.
        \par Therefore, we have 
        \begin{align*}
            \sup_i\mathrm{dist}\left(x_i, \mathrm{GH}(S\cap U)\cup[\tilde{U}, \partial_\infty\mathbb{H}^n\setminus U]\right)<\infty,
        \end{align*}
        where we denote $[A,B]=\bigcup_{a\in A, b\in B}[a, b]$. Let $s_0\in S\cap \tilde{U}\setminus\{s\}$ be arbitrary. Now observe that the function 
        \begin{align*}
            D:\hat{U}\times(\partial_\infty\mathbb{H}^n\setminus U)&\longrightarrow\mathbb{R}_{+}\\
            (\xi,\eta) &\longrightarrow \sup_{z\in [\xi,\eta]\cap \mathrm{CH}(\tilde{U})}\mathrm{dist}(z, [\xi, s_0])
        \end{align*}
        is continuous, and is thus bounded on the relatively compact subset $\tilde{U}\times(\partial_\infty\mathbb{H}^n\setminus U)\subset \hat{U}\times(\partial_\infty\mathbb{H}^n\setminus U)$. We then have for $i\geq N_0$,
        \begin{align*}
            \mathrm{dist}(x_i, \mathrm{GH}(S\cap U))\leq \mathrm{dist}(x_i, \mathrm{GH}(S\cap U)\cup[\tilde{U}, \partial_\infty\mathbb{H}^n\setminus U])+\sup_{\tilde{U}\times(\partial_\infty\mathbb{H}^n\setminus U)}D,
        \end{align*}
        and thus $\sup_i\mathrm{dist}(x_i, \mathrm{CH}(S\cap U))\leq\sup_i\mathrm{dist}(x_i, \mathrm{GH}(S\cap U))<\infty$. 
    \end{proof}
    \subsection{Proof of Theorem \ref{thm:main-admissible}}\label{subsec:pf-thm-main-admissible}
    For any $x\in N_D(C)$, there exists by Lemma \ref{lm:boundary-analysis} an angle $\theta=\theta(D)>0$ and $\xi\in \partial_\infty \mathbb{H}^n$ such that
    \begin{align*}
        \mathrm{Cone}(x\xi, \theta)\cap\partial_\infty \mathbb{H}^n\subseteq U.
    \end{align*}
    We claim that for all $R>R_0=R_0(D)$,  
    \begin{align}\label{eq:cone-containment}
        \mathrm{Cone}\left(x\xi, \frac{\theta}{12}\right)\cap\partial B_R(x)\subseteq\mathrm{CH}\left(\mathrm{Cone}(x\xi, \theta)\cap \partial_\infty X\right).
    \end{align}
    Note that (\ref{eq:cone-containment}) immediately shows admissibility of $\mathrm{CH}(U)$, so the rest of this subsection is devoted to showing (\ref{eq:cone-containment}).
    \par Let $y\in\mathrm{Cone}\left(x\xi, \frac{\theta}{12}\right)\cap\partial B_R(x)$ be arbitrary. Pick any point $\eta_1\in\partial_\infty \mathbb{H}^n$ such that $\frac{\theta}{6}<\measuredangle_x(\eta_1, \xi)<\frac{\theta}{3}$, and let $\eta_2\in\partial_\infty \mathbb{H}^n$ be such that $y\in[\eta_1, \eta_2]$. Then in particular we have 
    \begin{align*}
        \frac{\theta}{12}<\measuredangle_x(\eta_1, y)<\frac{5}{12}\theta.
    \end{align*}
    Claim \ref{claim:tiny-angle-side} below shows that, for $R$ large enough depending on $\theta$, we have 
    \begin{align*}
        \measuredangle_x(y, \eta_2)<\frac{\theta}{12}.
    \end{align*}
    Thus $\measuredangle_{x}(\eta_1, \eta_2)<\frac{\theta}{2}$, and hence $\eta_2\in\mathrm{Cone}(x\xi, \theta)\cap\partial_\infty \mathbb{H}^n$. Then the set $\mathrm{CH}\left(\mathrm{Cone}(x\xi,\theta)\cap\partial_\infty \mathbb{H}^n\right)$ contains the entire geodesic $[\eta_1, \eta_2]$, and hence also contains $y$.
    \begin{claim}\label{claim:tiny-angle-side}
        Let $x, y\in \mathbb{H}^n$ and $\xi, \eta\in\partial_\infty \mathbb{H}^n$ be such that $y\in [\xi, \eta]$. If $\measuredangle_x(\xi, y)=\alpha$ and $\mathrm{dist}(x, y)=R$, we have 
        \begin{align*}
            \measuredangle_x(y, \eta)\lesssim e^{-2R},
        \end{align*}
        where the implicit constant depends on $\alpha$. 
    \end{claim}
    \begin{proof}
        By the dual hyperbolic law of cosines applied to $xy\xi$ and to $xy\eta$, we see that 
        \begin{gather}
            1=-\cos\alpha \cos\measuredangle_{{y}}({\xi},{x})+\sin\alpha\sin\measuredangle_{{y}}({\xi}, {x})\cosh(R),\label{eq:cosines-1} \\ 
            1=\cos\beta\cos\measuredangle_{{y}}({\xi}, {x}) + \sin\beta\sin\measuredangle_{{y}}({\xi}, {x})\cosh(R).\label{eq:cosines-2}
        \end{gather}
        It follows from (\ref{eq:cosines-1}) that for large $R$, we have $\measuredangle_{{y}}({x}, {\xi})\lesssim e^{-R}$. Straightforward analysis of (\ref{eq:cosines-2}) then implies $\beta\lesssim \measuredangle_{{y}}({\xi},{x})^2\lesssim e^{-2R}$.
    \end{proof}
    \bibliographystyle{amsplain}
    \bibliography{main}
\end{document}